\documentclass{cimart_novo}

\usepackage{amsmath,amssymb,amsthm,amsfonts,enumerate, enumitem,hyperref,cleveref}
\usepackage{dsfont}
\usepackage[all]{xy}
\usepackage[T1]{fontenc}
\usepackage{tikz}
\usepackage{pgfplots}
\pgfplotsset{compat=1.15}
\pgfplotsset{compat=1.15}
\usepackage{mathrsfs}
\usetikzlibrary{arrows}

\DeclareMathOperator{\Aut}{Aut}

\DeclareMathOperator{\supp}{supp}

\DeclareMathOperator{\spn}{span}
\DeclareMathOperator{\Ann}{Ann}

\theoremstyle{plain}
\newtheorem{thrm}{Theorem}[section]
\newtheorem{cor}[thrm]{Corollary}
\newtheorem{prop}[thrm]{Proposition}
\newtheorem{lem}[thrm]{Lemma}

\theoremstyle{definition}
\newtheorem{defn}[thrm]{Definition}
\newtheorem{rem}[thrm]{Remark}
\newtheorem{exm}[thrm]{Example}

\crefrangeformat{equation}{#3(#1)#4--#5(#2)#6}

\crefname{thrm}{Theorem}{Theorems}
\crefname{theorem}{Theorem}{Theorems}
\crefname{lem}{Lemma}{Lemmas}
\crefname{cor}{Corollary}{Corollaries}
\crefname{prop}{Proposition}{Propositions}
\crefname{defn}{Definition}{Definitions}
\crefname{exm}{Example}{Examples}
\crefname{rem}{Remark}{Remarks}
\crefname{conj}{Conjecture}{Conjectures}
\crefname{quest}{Question}{Questions}
\crefname{section}{Section}{Sections}
\crefname{equation}{\unskip}{\unskip}
\crefname{enumi}{\unskip}{\unskip}
\crefname{subsection}{Subsection}{Subsections}

\renewcommand{\iff}{\Leftrightarrow}
\newcommand{\impl}{\Rightarrow}
\newcommand{\gen}[1]{\langle #1\rangle}

\newcommand{\G}{\Gamma}
\newcommand{\lb}{\lambda}
\newcommand{\sg}{\sigma}
\newcommand{\vf}{\varphi}

\newcommand{\A}{\mathcal{A}}

\newcommand{\Z}{\mathbb{Z}}

\newcommand{\sst}{\subseteq}

\newcommand{\id}{\mathrm{id}}

\newcommand{\ch}{\mathrm{char}}

\newcommand{\m}{{}^{-1}}

\VOLUME{33}
\ISSUE{3}
\YEAR{2025}
\NUMBER{6}
\DOI{https://doi.org/10.46298/cm.13990}

\title{% Please, capitalize only the first word
	$\sg$-matching and interchangeable structures on certain associative algebras
}

\author{% Please, use "Firstname Lastname" format, without abreviations
	Mykola Khrypchenko
}

\authorinfo{Universidade Federal de Santa Catarina, Brazil and Universidade do Porto, Portugal}{nskhripchenko@gmail.com}

\abstract{%
	We describe $\sg$-matching, interchangeable and, as a consequence, totally compatible products on some classes of associative algebras, including unital algebras, the semigroup algebras of rectangular bands, algebras with enough idempotents, free non-unital associative algebras and free non-unital commutative associative algebras.
}



\begin{document}
	
	% Here is where the main text should be typed:
	
	% A table of contents will be automatically inserted in your article if it
	% has 3 or more sections.  Please, do not try to manually change this
	% behaviour.
	
	% Also, please consider the following suggestions while preparing your 
	% manuscript (as they will speed up the editorial process):
	% * Avoid starting a new sentence with a mathematical formula;
	% * Try to separate adjacent formulas with words;
	% * Avoid inline formulas longer than half of a line. You can use math 
	%   displays (\[...\]) instead;
	% * Consider the use the enumerate and itemize environments for lists;
	% * Consider the use of \dots, \ldots, \dotsc, \cdot, etc, instead of "..." 
	%   or ".";
	% * Instead of numbering or citing an article by hand (using parenthesis or 
	%   brackets), consider the use of \cite, \ref and \eqref for citations and
	%   cross-references;
	% * Try to avoid inserting horizontal or vertical spacing, such as \hskip, 
	%   \vskip and \bigskip;
	% * Try to avoid inserting line or page brakes, such as \\, \newpage and
	%   \clearpage.
	
	\section*{Introduction}	
	Let $n$ be a positive integer and $\A$ a variety of algebras with one $n$-linear operation $[\cdot,\dots,\cdot]$. Consider two $n$-linear operations $[\cdot,\dots,\cdot]_1$ and $[\cdot,\dots,\cdot]_2$ on the same vector space $V$, such that the algebras $(V,[\cdot,\dots,\cdot]_1)$ and $(V,[\cdot,\dots,\cdot]_2)$ belong to $\A$. Then $[\cdot,\dots,\cdot]_1$ and $[\cdot,\dots,\cdot]_2$ are said to be \textit{compatible} if $(V,[\cdot,\dots,\cdot]_1+[\cdot,\dots,\cdot]_2)$ belongs to $\A$. In this case, the triple $(V,[\cdot,\dots,\cdot]_1,[\cdot,\dots,\cdot]_2)$ is called a \textit{compatible $\A$-algebra}. For example, there are compatible associative algebras~\cite{Carinena-Grabowski-Marmo2000,Odesskii-Sokolov06}, compatible Lie algebras~\cite{Golubchik-Sokolov02}, compatible pre-Lie algebras~\cite{Abdelwahab-Kaygorodov-Makhlouf24}, compatible Leibniz algebras~\cite{Makhlouf-Saha23} and so on. If the signature of a variety $\A$ has several operations (of possibly different types), one can consider the compatibility with respect to some fixed operation, for example, compatible Poisson brackets~\cite{Bolsinov91}, compatible Hom-Lie algebras~\cite{Das23CompHomLie}, compatible Hom-Lie triple systems~\cite{Teng-Long-Zhang-Lin23} and so on, or with respect to all the operations at the same time as in~\cite[Definition A]{Strohmayer08}. 
	
	The notion of a pair of compatible Poisson or Lie brackets seems to have firstly appeared in mathematical physics~\cite{Magri78,Reyman-Semenov89,Bolsinov91} and has been studied in this context for a couple of decades~\cite{Carinena-Grabowski-Marmo2000,Bolsinov-Borisov02,Golubchik-Sokolov02,Golubchik-Sokolov05}. A pair of compatible associative products was introduced (under the name ``quantum bi-Hamiltonian system'') and investigated from the purely algebraic point of view in~\cite{Carinena-Grabowski-Marmo2000}, where some general examples of such products were given. Odesski and Sokolov observed in~\cite{Odesskii-Sokolov06} that the description of pairs of compatible associative structures can be an interesting mathematical problem on its own. Thus, they managed to characterize in~\cite{Odesskii-Sokolov06} the compatible associative products on the matrix algebra $M_n(\mathbb{C})$ in terms of the so-called $n$-dimensional representations of $M$-algebras, and in the subsequent paper~\cite{Odesskii-Sokolov08} they used these products to construct a solution to the classical Yang-Baxter equation. Recently, there has appeared an interest in the algebraic classification of nilpotent compatible algebras of small dimensions~\cite{Abdelwahab-Kaygorodov-Makhlouf24,Ladra-Cunha-Lopes24}.
	
	At a more abstract level, compatible operations have been studied by specialists in operad theory. Dotsenko and  Khoroshkin calculated in~\cite{Dotsenko-Khoroshkin07} the dimensions of the graded components of the operads of a pair of compatible Lie brackets and compatible Poisson brackets. For the operad of a pair of compatible associative products, the corresponding dimensions have been calculated by Dotsenko in~\cite{Dotsenko09}. Strohmayer~\cite{Strohmayer08} gave a general way to construct the (binary quadratic) operad of two compatible structures from the (binary quadratic) operad of a single structure. He also pointed out in \cite[p. 2525]{Strohmayer08} three other kinds of compatibilities, the first of which corresponds to the interchangeability of the structures, the second one is a part of the matching compatibility and the third one is the total compatibility. Algebras with two matching associative products (called matching dialgebras) and their connection to other classes of algebras were studied in~\cite{zbg12}, while totally compatible associative products and the corresponding operads were investigated in~\cite{zbg13,Zhang13}. The notions of compatibility, matching compatibility and total compatibility were generalized to families of specific algebras (including the associative ones) in~\cite{zgg20} and to families of algebras over an arbitrary (unary binary quadratic/cubic) operad in~\cite{zgg24}. Finally, even more general matching compatibilities were introduced and studied in~\cite{zgg23}.
	
	Observe that, given an associative algebra $(A,\cdot)$, the compatible (with $\cdot$) products on $A$ are exactly the associative Hochschild $2$-cocycles of $A$ with values in $A$. However, even knowing $HH^2(A,A)$, it is difficult to characterize the associativity of the corresponding $2$-cocycles in simple terms (for example, see the case $HH^2(A,A)={0}$ in~\cite{Odesskii-Sokolov06}). Nevertheless, the description of matching and interchangeable products on $A$ seems to be a manageable task. Moreover, it turns out that for unital associative algebras such structures admit simple descriptions. In the non-unital case there are still interesting algebras to consider.
	
	In \cref{prelim} we specify the definitions of $\sg$-matching ($\sg\in S_2$), interchangeable and total compatibility to the case of a pair of associative products and give examples showing the difference between the compatibilities. As often happens, nilpotent algebras are a good source of such examples. The case of unital products is considered in \cref{sec-unital} (see \cref{id-match-unital-algebra,comp-prod-unital-algebra}) as a motivation for \cref{sec-idemp-alg}, where we prove some general results on idempotent algebras (see \cref{interchangeable-is-totally-comp-A^2=A,(12)-matching-is-totally-comp-left-right-unit}) and proceed to more specific classes of such algebras: the semigroup algebras of rectangular bands and algebras with enough idempotents (the latter class includes some algebras of infinite matrices, path algebras and infinite direct sums of unital algebras). The descriptions of $\sg$-matching and interchangeable products on these algebras are given in \cref{id-match-rect-band,(12)-matching-rect-band,id-matching-enough-idemp,comp-prod-A-with-enough-idemp}. \cref{sec-free-alg} is devoted to free non-unital associative algebras. We first prove a general result valid for all algebras without zero divisors (\cref{(12)-match=>id-match-without-0-div}) and then treat separately the cases of non-commutative and commutative non-unital free algebras (see \cref{id-matching-free-noncomm,totally-comp-free-noncomm,comp-prod-on-K[x],comp-prod-on-K[X]-with-|X|>1}).
	%	%\newpage	
	\section{Definitions and preliminaries}\label{prelim}
	
	All the algebras and vector spaces will be over a field $K$ and all the products will be binary and bilinear. 
	
	\subsection{Compatible structures}
	
	Given two bilinear binary operations $\cdot_1$ and $\cdot_2$ on a vector space $V$, their \textit{sum} $\star=\cdot_1+\cdot_2$ is defined by 
	\begin{align*}%\label{a*b=a-cdot_1-b+a-cdot_2-b}
		a\star b=a\cdot_1 b+a\cdot_2 b.
	\end{align*}
	Recall that two associative products $\cdot_1$ and $\cdot_2$ on $V$ are called \textit{compatible}, if $\cdot_1+\cdot_2$ is also associative. This is equivalent to the following equality:
	\begin{align}\label{compatible-ass-identity}
		(a\cdot_1 b)\cdot_2 c + (a\cdot_2 b)\cdot_1 c = a\cdot_1(b\cdot_2 c) + a\cdot_2(b\cdot_1 c)
	\end{align}
	for all $a,b,c\in V$. As it was observed in \cite[Remark on p. 4801]{Carinena-Grabowski-Marmo2000} and can be seen directly, $\cdot_1$ and $\cdot_2$ are compatible if and only if $\cdot_1$ is a Hochschild $2$-cocycle of $(V,\cdot_2)$ with values in $V$ (or, symmetrically, $\cdot_2$ is a Hochschild $2$-cocycle of $(V,\cdot_1)$ with values in $V$). 
	
	Let us consider the following particular cases of \cref{compatible-ass-identity}. 
	\begin{defn}
		Adopting the terminology of~\cite{zgg23} we say that two associative products $\cdot_1$ and $\cdot_2$ on $V$ are 
		\begin{enumerate}
			\item \textit{$\sg$-matching} (where $\sg\in S_2=\{\id,(12)\}$), if
			\begin{align}\label{matching-ass-identity}
				(a\cdot_1 b)\cdot_2c = a\cdot_{\sg(1)}(b\cdot_{\sg(2)}c) \text{ and } (a\cdot_2b)\cdot_1 c = a\cdot_{\sg(2)}(b\cdot_{\sg(1)}c),
			\end{align}
			\item \textit{totally compatible}, if 
			\begin{align}\label{totally-comp-ass-identity}
				(a\cdot_1 b)\cdot_2 c = (a\cdot_2 b)\cdot_1 c = a\cdot_1 (b\cdot_2 c) = a\cdot_2 (b\cdot_1 c),
			\end{align}
		\end{enumerate}
		for all $a,b,c\in V$.
	\end{defn} 
	Usually $\id$-matching $\cdot_1$ and $\cdot_2$ are simply called \textit{matching} in the literature. Observe that in this case $(V,\cdot_1,\cdot_2)$ is a \textit{matching dialgebra}~\cite{zbg12}, also known as \textit{$As^{(2)}$-algebra}~\cite{Zinbiel}. We will consider one more ``compatibility-type'' condition which makes sense for not necessarily associative products. 
	\begin{defn}
		We say that two (not necessarily associative) products $\cdot_1$ and $\cdot_2$ on $V$ are \textit{interchangeable} if
		\begin{align}\label{interchangeable-ass-identity}
			(a\cdot_1 b)\cdot_2c = (a\cdot_2 b)\cdot_1 c \text{ and } a\cdot_1(b\cdot_2 c) = a\cdot_2 (b\cdot_1 c).
		\end{align}
	\end{defn}
	
	\begin{rem}\label{tot-com-equiv-cond}
		For two associative products $\cdot_1$ and $\cdot_2$ on $V$ the following are equivalent:
		\begin{enumerate}
			\item\label{cdot_1-and-cdot_2-totally-comp} $\cdot_1$ and $\cdot_2$ are totally compatible;
			\item\label{cdot_1-and-cdot_2-sg-match-for-some-sg} $\cdot_1$ and $\cdot_2$ are interchangeable and at least one of the two equalities \cref{matching-ass-identity} holds for some $\sg\in S_2$;
			\item\label{sg-match-and-inter} $\cdot_1$ and $\cdot_2$ are $\sg$-matching for some $\sg\in S_2$ and at least one of the two equalities \cref{interchangeable-ass-identity} holds;
			\item\label{cdot_1-and-cdot_2-sg-match-for-all-sg} $\cdot_1$ and $\cdot_2$ are $\sg_1$-matching for some $\sg_1\in S_2$ and at least one of the two equalities \cref{matching-ass-identity} holds for $\sg\ne\sg_1$;
			%\item\label{cdot_1-and-cdot_2-sg-match-for-all-sg} $\cdot_1$ and $\cdot_2$ are $\sg$-matching for all $\sg\in S_2$.
		\end{enumerate}
		If $\ch(K)\ne 2$, then each of the conditions \cref{cdot_1-and-cdot_2-totally-comp,cdot_1-and-cdot_2-sg-match-for-some-sg,sg-match-and-inter,cdot_1-and-cdot_2-sg-match-for-all-sg} is also equivalent to 
		\begin{enumerate}
			\setcounter{enumi}{4}
			\item $\cdot_1$ and $\cdot_2$ are interchangeable and compatible.
		\end{enumerate} 
	\end{rem}

	Let $(A,\cdot)$ be an associative algebra. By a \textit{compatible} (resp. \textit{$\sg$-matching}, \textit{interchangeable} or \textit{totally compatible}) \textit{structure} on $A$ we mean an associative product $*$ on $A$ that is compatible (resp. $\sg$-matching, interchangeable or totally compatible) with $\cdot$. 
	
	The following examples show that the classes of $\id$-matching, $(12)$-matching and interchangeable structures on $(A,\cdot)$ are different and no class is in general contained in another one (although all of them contain the totally compatible structures). %Observe that $(A,\cdot)$ in the examples below is nilpotent.
	\begin{exm}
		Let $A$ be the $3$-dimensional nilpotent associative algebra with a basis $\{e_i\}_{i=1}^3$ and the multiplication table $e_1\cdot e_2=e_3$. 
		
		The product $*$ on $A$ given by $e_1*e_1=e_1$ and $e_1*e_2=e_2$ is clearly associative (observe that $(A,*)$ is isomorphic to the direct sum of a $1$-dimensional algebra with zero multiplication and the subalgebra of $M_2(K)$ generated by the matrix units $\{E_{11},E_{12}\}$). Then $(e_i*e_j)\cdot e_k\ne 0\iff e_i\cdot(e_j*e_k)\ne 0\iff (i,j,k)=(1,1,2)$, in which case $(e_i*e_j)\cdot e_k=e_i\cdot(e_j*e_k)=e_3$. Moreover, $(A\cdot A)*A=A*(A\cdot A)=\{0\}$. Thus, the product $*$ is $(12)$-matching with $\cdot$, but it is neither $\id$-matching nor interchangeable with $\cdot$ because $(A*A)\cdot A\ne \{0\}$.
		
		On the other hand, consider the associative product $\star$ on $A$ given by $e_1\star e_1=e_1$ and $e_1\star e_3=e_3$. Then, 
		\[
		(e_i\star e_j)\cdot e_k\ne 0
		\quad \textup{if and only if} \quad
		e_i\star(e_j\cdot e_k)\ne 0
		\quad \textup{if and only if} \quad
		(i,j,k)=(1,1,2),
		\]
		in which case $(e_i\star e_j)\cdot e_k=e_i\star(e_j\cdot e_k)=e_3$. Moreover, $(A\cdot A)\star A=A\cdot(A\star A)=\{0\}$. Thus, the product $\star$ is $\id$-matching with $\cdot$, but it is neither $(12)$-matching nor interchangeable with $\cdot$ because $(A\star A)\cdot A\ne \{0\}$.
	\end{exm}
	
	\begin{exm}
		Let $A$ be the algebra with a basis $\{e_i\}_{i=1}^6$ and multiplication table 
		\[  e_1\cdot e_2 = e_4, \: e_1\cdot e_5 = e_6, \: e_4\cdot e_3 = e_6, \: e_2\cdot e_3 = e_5.\]
		It is a nilpotent associative algebra isomorphic to the subalgebra of $M_4(K)$ generated by the matrix units $\{E_{12},E_{23},E_{34}\}$.
		
		The product $*$ on $A$ given by $e_1*e_2 = e_5$, $e_1*e_4 = e_6$ is clearly associative because $(A*A)*A=A*(A*A)=\{0\}$. Furthermore, 
		\[e_i*(e_j\cdot e_k)\ne 0\iff e_i\cdot(e_j*e_k)\ne 0\iff(i,j,k)=(1,1,2),\]
		in which case $e_i*(e_j\cdot e_k)=e_i\cdot(e_j*e_k)=e_6$. Moreover, $(A*A)\cdot A=(A\cdot A)*A=\{0\}$. Thus, $*$ is interchangeable with $\cdot$, but it is neither $\id$-matching nor $(12)$-matching with $\cdot$ because $A*(A\cdot A)\ne \{0\}$.
	\end{exm}
	
	Given two structures $*_1$ and $*_2$ on $A$ that are compatible with $\cdot$, we say that $*_1$ and $*_2$ are \textit{isomorphic}, if there exists an automorphism $\phi$ of $(A,\cdot)$ such that 
	\begin{align*}
		\phi(a*_1b)=\phi(a)*_2\phi(b)
	\end{align*}
	for all $a,b\in A$ (i.e. if $(A,\cdot,*_1)$ and $(A,\cdot,*_2)$ are isomorphic).
	
	\section{\texorpdfstring{$\sg$}{σ}-matching and interchangeable structures on unital associative algebras}\label{sec-unital}
	Although unital associative algebras are a subclass of algebras with enough idempotents considered in \cref{sec-idemp-alg}, we begin with this particular case to motivate the choice of the context for \cref{sec-idemp-alg}. We also introduce here some notions that will be used below.
	
	For any associative algebra $(A,\cdot)$ and fixed $x\in A$, one defines the \textit{mutation~\cite{ElduqueMyung} of $\cdot$ by $x$} to be the product $\cdot_x$ on $A$ given by $a\cdot_x b=a\cdot x\cdot b$ for all $a,b\in A$. The following should be well-known (cf. \cite[Formula (9)]{Carinena-Grabowski-Marmo2000} and \cite[Example 2.2]{Odesskii-Sokolov06}), but we couldn't find an explicit proof.
	\begin{lem}\label{mutation-id-matching}
		Let $(A,\cdot)$ be an associative algebra. For any $x\in A$ the product $\cdot_x$ is associative and $\id$-matching with $\cdot$.
	\end{lem}
	\begin{proof}
		For $a,b,c\in A$ we have $(a\cdot_x b)\cdot c=a\cdot x\cdot b\cdot c=a\cdot_x(b\cdot c)$ and $(a\cdot b)\cdot_x c=a\cdot b\cdot x\cdot c=a\cdot(b\cdot_x c)$. Moreover, $(a\cdot_x b)\cdot_x c=a\cdot x\cdot b\cdot x\cdot c=a\cdot_x (b\cdot_x c)$.
	\end{proof}
	
	For unital associative algebras the converse also holds.
	\begin{prop}\label{id-match-unital-algebra}
		Let $(A,\cdot)$ be a unital associative algebra. Then the $\id$-matching structures on $A$ are exactly the mutations of $\cdot$.
	\end{prop}
	\begin{proof}
		Let $1$ be the identity element of $(A,\cdot)$ and $*$ an $\id$-matching associative product on $A$. Then, for all $a,b\in A$ we have
		\[
		a*b=(a\cdot 1)*b=a\cdot(1*b)=a\cdot(1*(1\cdot b))=a\cdot((1*1)\cdot b)=a\cdot_{1*1}b.
		\qedhere
		\]
	\end{proof}
	
	Let $(A,\cdot)$ be a (not necessarily associative) algebra. Recall that the \textit{centroid} of $A$ is the space $\G(A)$ of linear maps $\vf:A\to A$ such that
	\begin{align*}%\label{x.vf(y)=vf(xy)=vf(x)y}
		x\cdot\vf(y)=\vf(x\cdot y)=\vf(x)\cdot y
	\end{align*}
	for all $x,y\in A$. 
	
	\begin{lem}\label{a*b=vf(ab)-tot-comp}
		Let $(A,\cdot)$ be a (not necessarily associative) algebra. Given $\vf\in\G(A)$, the product 
		\begin{align}\label{a*b=vf(a-cdot-b)}
			a*_\vf b:=\vf(a\cdot b)
		\end{align}
		is interchangeable with $\cdot$. Moreover, if $\cdot$ is associative, then $*_\vf$ is also associative and totally compatible with $\cdot$.
	\end{lem}
	\begin{proof}
		For all $a,b,c\in A$ we have 
		\begin{align*}
			(a*_\vf b)\cdot c&=\vf(a\cdot b)\cdot c=\vf((a\cdot b)\cdot c)=(a\cdot b)*_\vf c \text { and }\\
			a*_\vf (b\cdot c)&=\vf(a\cdot (b\cdot c))=a\cdot \vf(b\cdot c)=a\cdot (b*_\vf c).
		\end{align*}%$(a*_\vf b)\cdot c=\vf(a\cdot b)\cdot c=\vf((a\cdot b)\cdot c)=(a\cdot b)*_\vf c$ and $a*_\vf (b\cdot c)=\vf(a\cdot (b\cdot c))=a\cdot \vf(b\cdot c)=a\cdot (b*_\vf c)$.
		
		The second statement follows from \cite[Proposition 2.7]{zbg13}.
	\end{proof}
	
	\begin{defn}
		The product of the form \cref{a*b=vf(a-cdot-b)} is said to be \textit{determined} by $\vf\in\G(A)$.
	\end{defn}
	
	In the associative case any element $c$ of the \textit{center} $C(A)$ of $A$ defines $\vf\in\G(A)$ by means of $\vf(a)=c\cdot a$. If, moreover, $A$ is unital, then this gives an isomorphism of $K$-spaces $C(A)\cong\G(A)$, whose inverse maps $\vf\in\G(A)$ to $\vf(1)$. Thus, we have the following.
	\begin{lem}
		Let $(A,\cdot)$ be a unital associative algebra. Then the products $*$ determined by elements of $\G(A)$ are exactly the mutations of $\cdot$ by elements of $C(A)$.
	\end{lem} 
	
	\begin{prop}\label{comp-prod-unital-algebra}
		Let $(A,\cdot)$ be a unital associative algebra and $*$ be an associative product on $A$. Then the following are equivalent:
		\begin{enumerate}
			\item\label{*-(12)-matching-with-1} $*$ is $(12)$-matching with $\cdot$;
			\item\label{*-interchangeable-with-1} $*$ is interchangeable with $\cdot$;
			\item\label{*-totally-compatible-with-1} $*$ is totally compatible with $\cdot$;
			\item\label{*-mutation-with-1} $*$ is determined by an element of $\G(A)\cong C(A)$.
		\end{enumerate}
	\end{prop}
	\begin{proof}
		\cref{*-(12)-matching-with-1}$\impl$\cref{*-mutation-with-1}. Let $*$ be $(12)$-matching with $\cdot$. Then for all $a\in A$ we have 
		\[1*a=(1*a)\cdot 1=1\cdot (a*1)=a*1.\] 
		It follows that 
		\[(1*1)\cdot a=1\cdot (1*a)=1*a=a*1=(a*1)\cdot 1=a\cdot(1*1),\]
		so $1*1\in C(A)$. Using these equalities, we also have 
		\[a*b=(a*b)\cdot 1=a\cdot (b*1)=a\cdot ((1*1)\cdot b)=a\cdot_{1*1} b\] 
		for all $a,b\in A$.
		
		\cref{*-interchangeable-with-1}$\impl$\cref{*-mutation-with-1}. Let $*$ be interchangeable with $\cdot$. Then for all $a\in A$ we have   
		\[1*a=(1\cdot 1)*a=(1*1)\cdot a, \text{ and } a*1=a*(1\cdot 1)=a\cdot(1*1).\] Now, $1*a=(1*a)\cdot 1=(1\cdot a)*1=a*1$, whence $(1*1)\cdot a=a\cdot(1*1)$, i.e. $1*1\in C(A)$. Thus, 
		\[a*b=(a\cdot 1)*b=(a*1)\cdot b=a\cdot(1*1)\cdot b=a\cdot_{1*1} b\] for all $a,b\in A$.
		
		The implications \cref{*-mutation-with-1}$\impl$\cref{*-totally-compatible-with-1}, \cref{*-totally-compatible-with-1}$\impl$\cref{*-interchangeable-with-1} and \cref{*-totally-compatible-with-1}$\impl$\cref{*-(12)-matching-with-1} are obvious.
	\end{proof}
	
	\begin{cor}\label{mutation-by-non-central-element}
		Let $A$ be a non-commutative unital associative algebra. Then there exist products on $A$ that are $\id$-matching with $\cdot$, but not totally compatible with $\cdot$.
	\end{cor}
	\begin{proof}
		Let $x$ be a non-central element of $A$. Then $\cdot_x$ is $\id$-matching with $\cdot$ by \cref{mutation-id-matching}. If $\cdot_x$ was totally compatible with $\cdot$, then by \cref{comp-prod-unital-algebra} there would be $c\in C(A)$ such that $\cdot_x=\cdot_c$, whence $x=1\cdot_x 1=1\cdot_c 1=c$, a contradiction.
	\end{proof}
	
	\begin{rem}
		Let $(A,\cdot)$ be a unital associative algebra and $x,y\in A$. Then the structures $\cdot_x$ and $\cdot_y$ are isomorphic if and only if there is $\phi\in\Aut(A)$ such that $\phi(x)=y$.
		
		For, given $\phi\in\Aut(A)$, one has \[\phi(a\cdot_xb)=\phi(a)\cdot_y\phi(b)\iff \phi(a)\cdot\phi(x)\cdot\phi(b)=\phi(a)\cdot y\cdot \phi(b).\] The latter holds for all $a,b\in A$ if and only if $\phi(x)=y$ (take $a=b=1$).
	\end{rem}
	
	\begin{rem}
		Whenever $\G(A)=K$, there are only two non-isomorphic structures of the form \cref{a*b=vf(a-cdot-b)}: the original product $\cdot$ and the zero product.
		
		For, in this case $a*_\vf b=\lb(a\cdot b)$ for some $\lb\in K$, so, if $\lb\ne 0$, the map $\lb\m\id$ is an isomorphism between $\cdot$ and $*_\vf$.
	\end{rem}
	
	\begin{rem}
		For non-unital algebras the result of \cref{comp-prod-unital-algebra} may be false. For example, if $A$ is an algebra with zero multiplication, then any associative product $*$ on $A$ is totally compatible with $\cdot$ since all the mixed monomials in \cref{totally-comp-ass-identity} are zero. However, all the mutations of $\cdot$ and all the products on $A$ determined by elements of $\G(A)$ are zero.
	\end{rem}
	
	\section{\texorpdfstring{$\sg$}{σ}-matching and interchangeable structures on certain idempotent associative algebras}\label{sec-idemp-alg}
	
	In this section we are going to see to which extent the results of \cref{sec-unital} generalize to several classes of associative algebras that are in some sense close to being unital. Recall that a (not-necessarily associative) algebra $(A,\cdot)$ is said to be \textit{idempotent}\footnote{In the context of Lie algebras one prefers to use the term ``perfect'' rather than ``idempotent''.} if $A\cdot A=A$. %For such a class of algebras we have the following general result.
	\begin{prop}\label{interchangeable-is-totally-comp-A^2=A}
		Let $(A,\cdot)$ be an idempotent associative algebra and $*$ be an associative product on $A$. Then the following are equivalent:
		\begin{enumerate}
			\item\label{*-interchangeable-A^2=A} $*$ is interchangeable with $\cdot$;
			\item\label{*-totally-compatible-A^2=A} $*$ is totally compatible with $\cdot$.
		\end{enumerate}
	\end{prop}
	\begin{proof}
		We only need to prove \cref{*-interchangeable-A^2=A}$\impl$\cref{*-totally-compatible-A^2=A}. Let $*$ be interchangeable with $\cdot$. Then, for all $a,b,c\in A$ with $b=b_1\cdot b_2$ we have \[a\cdot(b*c)=a\cdot((b_1\cdot b_2)*c)=a\cdot((b_1*b_2)\cdot c)=(a\cdot(b_1*b_2))\cdot c=(a*(b_1\cdot b_2))\cdot c=(a*b)\cdot c.\] Since $A\cdot A=A$, then by linearity $a\cdot(b*c)=(a*b)\cdot c$ for all $a,b,c\in A$. Thus, $*$ is  totally compatible with $\cdot$ by \cref{tot-com-equiv-cond}\cref{cdot_1-and-cdot_2-sg-match-for-some-sg}.
	\end{proof}
	
	It is natural to ask if one can replace interchangeable products by $\sg$-matching ones in \cref{interchangeable-is-totally-comp-A^2=A}. For $\sg=\id$ the answer is ``no'' even in the case of unital algebras, as we saw in \cref{mutation-by-non-central-element}. For $\sg=(12)$ there are classes of idempotent algebras containing the unital ones for which the answer is positive and those for which it is negative. We first point out a class admitting the positive answer.
	
	\begin{prop}\label{(12)-matching-is-totally-comp-left-right-unit}
		Let $(A,\cdot)$ be an associative algebra with a left or right unit and $*$ be an associative product on $A$. Then the following are equivalent:
		\begin{enumerate}
			\item\label{*-(12)-matching-left-unit} $*$ is $(12)$-matching with $\cdot$;
			\item\label{*-totally-compatible-left-unit} $*$ is totally compatible with $\cdot$.
		\end{enumerate}
	\end{prop}
	\begin{proof}
		We only need to prove \cref{*-(12)-matching-left-unit}$\impl$\cref{*-totally-compatible-left-unit}. Assume that $A$ has a left unit $e$. Let $*$ be $(12)$-matching with $\cdot$. Then for all $a,b,c\in A$ we have 
		\begin{align*}
			(a*b)\cdot c
			&=e\cdot((a*b)\cdot c)\\
			&=(e\cdot(a*b))\cdot c\\
			&=((e*a)\cdot b)\cdot c\\
			&=(e*a)\cdot(b\cdot c)\\
			&=e\cdot(a*(b\cdot c))\\
			&=a*(b\cdot c).
		\end{align*}
		It follows that $*$ is totally compatible with $\cdot$ by \cref{tot-com-equiv-cond}\cref{cdot_1-and-cdot_2-sg-match-for-all-sg}. If $A$ has a right unit $e$, then we symmetrically have 
		\begin{align*}
			a\cdot(b*c)
			&=(a\cdot(b*c))\cdot e\\
			&=a\cdot((b*c)\cdot e)\\
			&=a\cdot(b\cdot(c*e))\\
			&=(a\cdot b)\cdot(c*e)\\
			&=((a\cdot b)*c)\cdot e=(a\cdot b)*c.
			\qedhere
		\end{align*}
	\end{proof}
	
	\subsection{The semigroup algebra of a rectangular band}\label{sec-rect-band}
	
	Given arbitrary non-empty sets $I$ and $J$, one easily sees that $S:=\{e_{ij}\mid (i,j)\in I\times J\}$ is a semigroup under the multiplication 
	\begin{align}\label{e_ij.e_kl=e_il}
		e_{ij}\cdot e_{kl}=e_{il}
	\end{align}
	for all $(i,j),(k,l)\in I\times J$. It is called \textit{a rectangular band}~\cite{Clifford-Preston-1}, and it is a classical example of a semigroup in which every element is idempotent. Let $(A,\cdot)$ be the semigroup $K$-algebra of $S$. Thus, $A$ is an associative algebra admitting a basis consisting of the idempotents $e_{ij}, (i,j)\in I\times J$. We are going to describe the $\sg$-matching structures on $(A,\cdot)$.
	
	\begin{prop}\label{id-match-rect-band}
		The $\id$-matching structures on $(A,\cdot)$ are exactly the products $*$ of the form
		\begin{align}\label{e_ij*e_kl=lb_jk.e_il}
			e_{ij}*e_{kl}=\lb_{jk} e_{il},
		\end{align}
		where $\lb_{jk}\in K$ and $(i,j),(k,l)\in I\times J$.
	\end{prop}
	\begin{proof}
		Let $*$ be $\id$-matching with $\cdot$. Then for all $(i,j),(k,l)\in I\times J$ we have 
		\[e_{ij}*e_{kl}=(e_{ij}\cdot e_{ij})*e_{kl}=e_{ij}\cdot (e_{ij}*e_{kl})=e_{ij}\cdot (e_{ij}*(e_{kl}\cdot e_{kl}))=e_{ij}\cdot (e_{ij}*e_{kl})\cdot e_{kl},\]
		which equals $\mu_{ij,kl}e_{il}$ for some $\mu_{ij,kl}\in K$ by \cref{e_ij.e_kl=e_il}. Moreover, for any $(p,q)\in I\times J$ we have 
		\[\mu_{ij,kl}e_{iq}=(e_{ij}*e_{kl})\cdot e_{pq}=e_{ij}*(e_{kl}\cdot e_{pq})=e_{ij}*e_{kq}=\mu_{ij,kq}e_{iq},\] whence $\mu_{ij,kl}=\mu_{ij,kq}$. Similarly, \[\mu_{ij,kl}e_{pl}=e_{pq}\cdot(e_{ij}*e_{kl})=(e_{pq}\cdot e_{ij})*e_{kl}=e_{pj}*e_{kl}=\mu_{pj,kl}e_{pl},\] whence $\mu_{ij,kl}=\mu_{pj,kl}$. Thus, $\mu_{ij,kl}$ does not depend on $i$ and $l$, so denoting it by $\lb_{jk}$, we get \cref{e_ij*e_kl=lb_jk.e_il}.
		
		Conversely, let $*$ be given by \cref{e_ij*e_kl=lb_jk.e_il}. Then for all $(i,j),(k,l),(p,q)\in I\times J$ we have 
		\[
		(e_{ij}*e_{kl})*e_{pq}=\lb_{jk} e_{il}*e_{pq}=\lb_{jk}\lb_{lp}e_{iq}
		\quad \textup{and} \quad
		e_{ij}*(e_{kl}*e_{pq})=\lb_{lp}e_{ij}*e_{kq}=\lb_{lp}\lb_{jk}e_{iq},
		\]
		so that $*$ is associative. Since \[(e_{ij}*e_{kl})\cdot e_{pq}=\lb_{jk} e_{il}\cdot e_{pq}=\lb_{jk}e_{iq}=e_{ij}*e_{kq}=e_{ij}*(e_{kl}\cdot e_{pq})\] and similarly \[(e_{ij}\cdot e_{kl})*e_{pq}=e_{il}*e_{pq}=\lb_{lp}e_{iq}=e_{ij}\cdot \lb_{lp}e_{kq}=e_{ij}\cdot(e_{kl}*e_{pq}),\] the product $*$ is $\id$-matching with $\cdot$.
	\end{proof}
	
	\begin{rem}
		The product~\cref{e_ij*e_kl=lb_jk.e_il} is a mutation of $\cdot$ if and only if $\lb_{jk}$ does not depend on $j$ and $k$.
		
		Indeed, given $a=\sum_{p,q} a_{pq}e_{pq}$ with $a_{pq}\in K$, we have $e_{ij}\cdot_a e_{kl}=e_{ij}\cdot a\cdot e_{kl}=\lb e_{il}$, where $\lb=\sum_{p,q}a_{pq}$. Conversely, if $\lb_{jk}=\lb$ for all $j$ and $k$ in~\cref{e_ij*e_kl=lb_jk.e_il}, then $e_{ij}*e_{kl}=\lb e_{il}=e_{ij}\cdot_a\cdot e_{kl}$, where $a=\lb e_{pq}$ for some fixed arbitrary $(p,q)\in I\times J$.
	\end{rem}
	
	\begin{lem}\label{e_ij*e_kl-is-e_il*e_il}
		Let $*$ be an associative product on $A$. If $*$ is $(12)$-matching with $\cdot$, then 
		\begin{align}\label{e_ij*e_kl=e_il*e_il}
			e_{ij}*e_{kl}=e_{il}*e_{il}
		\end{align}
		for all $(i,j),(k,l)\in I\times J$.
	\end{lem}
	\begin{proof}
		We have
		\begin{align*}
			e_{ij}*e_{kl}
			&=(e_{ij}\cdot e_{ij})*e_{kl}\\
			&=e_{ij}*(e_{ij}\cdot e_{kl})\\
			&=e_{ij}*e_{il}\\
			&=e_{ij}*(e_{il}\cdot e_{il})\\
			&=(e_{ij}\cdot e_{il})*e_{il}\\
			&=e_{il}*e_{il}.
			\qedhere
		\end{align*}
	\end{proof}
	
	Denote by $\Ann(A)$ the (two-sided) \textit{annihilator} of $A$, i.e. the ideal of $A$ consisting of $a\in A$ such that $a\cdot b=b\cdot a=0$ for all $b\in A$. 
	
	\begin{lem}\label{centralizer-e_ij}
		For arbitrary $(i,j)\in I\times J$ the centralizer of $e_{ij}$ in $A$ coincides with $\spn_K\{e_{ij}\}\oplus \Ann(A)$.
	\end{lem}
	\begin{proof}
		It is obvious that $\spn_K\{e_{ij}\}\cap \Ann(A)=\{0\}$ and that any element of the sum $\spn_K\{e_{ij}\}+\nolinebreak\Ann(A)$ commutes with $e_{ij}$, so we only need to prove that the centralizer of $e_{ij}$ is contained in $\spn_K\{e_{ij}\}+\Ann(A)$. Let $a=\sum_{k,l} a_{kl}e_{kl}\in A$ with $a_{kl}\in K$. Then $a\cdot e_{ij}=\sum_{k,l}a_{kl}e_{kj}=\sum_{k\in I}\left(\sum_{l\in J}a_{kl}\right)e_{kj}$ and $e_{ij}\cdot a=\sum_{k,l}a_{kl}e_{il}=\sum_{l\in J}\left(\sum_{k\in I}a_{kl}\right)e_{il}$. Thus,
		\begin{align}\label{a.e_ij=e_ij.a}
			a\cdot e_{ij}=e_{ij}\cdot a\iff
			\begin{cases}
				\sum_{l\in J}a_{kl}=0\text{ for all }k\ne i,\\
				\sum_{k\in I}a_{kl}=0\text{ for all }l\ne j,\\
				\sum_{l\in J}a_{il}=\sum_{k\in I}a_{kj}.
			\end{cases}
		\end{align}
		Assuming that $a\cdot e_{ij}=e_{ij}\cdot a$, for all $(p,q)\in I\times J$ by \cref{a.e_ij=e_ij.a} we have \[a\cdot e_{pq}=\sum_{k\in I}\left(\sum_{l\in J}a_{kl}\right)e_{kq}=\lb e_{iq}\] and \[e_{pq}\cdot a=\sum_{l\in J}\left(\sum_{k\in I}a_{kl}\right)e_{pl}=\lb e_{pj},\] where $\lb=\sum_{l\in J}a_{il}=\sum_{k\in I}a_{kj}$.  Then $(a-\lb e_{ij})\cdot e_{pq}=a\cdot e_{pq}-\lb e_{iq}=0$ and  $e_{pq}\cdot(a-\lb e_{ij})=e_{pq}\cdot a-\lb e_{pj}=0$, so that $a-\lb e_{ij}\in\Ann(A)$.
	\end{proof}
	
	\begin{lem}\label{e_ij*e_ij=centroid+ann}
		Let $*$ be an associative product on $A$. If $*$ is $(12)$-matching with $\cdot$, then there exist $\lb\in K$ and $r_{ij}\in\Ann(A)$, such that
		\begin{align}\label{e_ij*e_ij=lb.e_ij+r_ij}
			e_{ij}*e_{ij}=\lb e_{ij}+r_{ij}
		\end{align}
		for all $(i,j)\in I\times J$.
	\end{lem}
	\begin{proof}
		Since $e_{ij}\cdot(e_{ij}*e_{ij})=(e_{ij}*e_{ij})\cdot e_{ij}$, by \cref{centralizer-e_ij} there are $\lb_{ij}\in K$ and $r_{ij}\in\Ann(A)$, such that $e_{ij}*e_{ij}=\lb_{ij} e_{ij}+r_{ij}$. Now, on the one hand $(e_{ij}*e_{ij})\cdot e_{il}=(\lb_{ij} e_{ij}+r_{ij})\cdot e_{il}=\lb_{ij}e_{il}$, and on the other hand using \cref{e_ij*e_kl=e_il*e_il} we have 
		\[(e_{ij}*e_{ij})\cdot e_{il}=e_{ij}\cdot (e_{ij}* e_{il})=e_{ij}\cdot (e_{il}* e_{il})=e_{ij}\cdot(\lb_{il} e_{il}+r_{il})=\lb_{il}e_{il},\]
		whence $\lb_{ij}=\lb_{il}$ for all $i\in I$ and $j,l\in J$. Similarly, $e_{kj}\cdot (e_{ij}*e_{ij})=e_{kj}\cdot(\lb_{ij} e_{ij}+r_{ij})=\lb_{ij}e_{kj}$ and 
		\[e_{kj}\cdot (e_{ij}*e_{ij})=(e_{kj}*e_{ij})\cdot e_{ij}=(e_{kj}*e_{kj})\cdot e_{ij}=(\lb_{kj} e_{kj}+r_{kj})\cdot e_{ij}=\lb_{kj} e_{kj}\] imply $\lb_{ij}=\lb_{kj}$ for all $i,k\in I$ and $j\in J$. Thus, $\lb_{ij}=\lb_{kl}$ for all $(i,j),(k,l)\in I\times J$, and \cref{e_ij*e_ij=lb.e_ij+r_ij} follows.
	\end{proof}
	
	\begin{prop}\label{(12)-matching-rect-band}
		A bilinear product $*$ on $A$ is a $(12)$-matching structure on $(A,\cdot)$ if and only if $*$ is associative and, for all $(i,j),(k,l)\in I\times J$,
		\begin{align}\label{e_ij*e_kl=lb.e_il+r_il}
			e_{ij}*e_{kl}=\lb e_{il}+r_{il},
		\end{align}
		where $\lb\in K$ and $r_{il}\in\Ann(A)$.
	\end{prop}
	\begin{proof}
		By \cref{e_ij*e_kl-is-e_il*e_il,e_ij*e_ij=centroid+ann} any $(12)$-matching structure on $(A,\cdot)$ has the form \cref{e_ij*e_kl=lb.e_il+r_il}. Conversely, let $*$ be a product on $A$ given by \cref{e_ij*e_kl=lb.e_il+r_il}. Then for all $(i,j),(k,l),(p,q)\in I\times J$ we have $(e_{ij}*e_{kl})\cdot e_{pq}=(\lb e_{il}+r_{il})\cdot e_{pq}=\lb e_{iq}$ and $e_{ij}\cdot(e_{kl}*e_{pq})=e_{ij}\cdot(\lb e_{kq}+r_{kq})=\lb e_{iq}$. Moreover, $(e_{ij}\cdot e_{kl})*e_{pq}=e_{il}*e_{pq}=\lb e_{iq}+r_{iq}$ and $e_{ij}*(e_{kl}\cdot e_{pq})=e_{ij}*e_{kq}=\lb e_{iq}+r_{iq}$.
	\end{proof}
	
	\begin{cor}\label{id-matching-iff-r_il=0}
		The totally compatible structures on $(A,\cdot)$ are exactly the products $*$ of the form
		\begin{align}\label{e_ij*e_kl=lb.e_il}
			e_{ij}*e_{kl}=\lb e_{il},
		\end{align}
		where $\lb\in K$ and $(i,j),(k,l)\in I\times J$.
	\end{cor}
	\begin{proof}
		Let $*$ be totally compatible with $\cdot$. In particular, $*$ is $(12)$-matching with $\cdot$. Then for all $(i,j)\in I\times J$ by \cref{e_ij*e_kl=lb.e_il+r_il} we have $(e_{ij}*e_{ij})\cdot e_{ij}=\lb e_{ij}$, while $e_{ij}*(e_{ij}\cdot e_{ij})=e_{ij}*e_{ij}=\lb e_{ij}+r_{ij}$. Thus, $r_{ij}=0$ for all $(i,j)\in I\times J$ in \cref{e_ij*e_kl=lb.e_il+r_il}, and we get \cref{e_ij*e_kl=lb.e_il}. Conversely, the product \cref{e_ij*e_kl=lb.e_il} is $*_\vf$, where $\vf=\lb\id\in\G(A)$, so $*$ is totally compatible with $\cdot$ by \cref{a*b=vf(ab)-tot-comp}.
	\end{proof}
	
	\begin{cor}
		We have $\G(A)=K$.
	\end{cor}
	\begin{proof}
		Let $\vf\in\G(A)$ and consider the totally compatible product $*_\vf$ on $A$ as in \cref{a*b=vf(ab)-tot-comp}. By \cref{id-matching-iff-r_il=0} there exists $\lb\in K$ such that $\vf(e_{ij})=\vf(e_{ij}\cdot e_{ij})=e_{ij}*_\vf e_{ij}=\lb e_{ij}$ for all $(i,j)\in I\times J$. By linearity, $\vf=\lb\id$.
	\end{proof}
	
	We summarize the results on totally compatible structures on $A$ in the following.
	\begin{prop}\label{totally-comp-rect-band}
		Let $*$ be an associative product on $A$. Then the following are equivalent:
		\begin{enumerate}
			\item\label{*-interchangeable-rect-band} $*$ is interchangeable with $\cdot$;
			\item\label{*-totally-compatible-rect-band} $*$ is totally compatible with $\cdot$;
			\item\label{*-mutation-rect-band} $*$ is determined by some $\vf\in\G(A)$.
		\end{enumerate}
	\end{prop}
	\begin{proof}
		This is a consequence of \cref{interchangeable-is-totally-comp-A^2=A,a*b=vf(ab)-tot-comp,id-matching-iff-r_il=0}.
	\end{proof}

	In general, the annihilator part of \cref{e_ij*e_kl=lb.e_il+r_il} may be non-trivial, which makes it difficult to characterize explicitly the associativity of \cref{e_ij*e_kl=lb.e_il+r_il} in terms of $\lb$ and $r_{il}$. So, we just give an example of such an associative product below, which also provides a $(12)$-matching structure on $A$ that is not totally compatible with $\cdot$, showing that one cannot replace interchangeable structures by $(12)$-matching ones in \cref{interchangeable-is-totally-comp-A^2=A}.
	\begin{exm}
		Let $I=J=\{1,2\}$ and $A$ be the semigroup algebra of the rectangular band $I\times J$. Then $\Ann(A)=\spn_K\{e_{11}-e_{12}-e_{21}+e_{22}\}$. Indeed, if $a=\sum_{i,j}a_{ij}e_{ij}$ with $a_{ij}\in K$ belongs to $\Ann(A)$, then $0=a\cdot e_{11}=(a_{11}+a_{12})e_{11}+(a_{21}+a_{22})e_{21}$ and $0=e_{11}\cdot a=(a_{11}+a_{21})e_{11}+(a_{12}+a_{22})e_{12}$, whence $a_{11}=-a_{12}=-a_{21}=a_{22}$, so that $a\in \spn_K\{e_{11}-e_{12}-e_{21}+e_{22}\}$. Conversely, $e_{11}-e_{12}-e_{21}+e_{22}\in\Ann(A)$ by a straightforward calculation.

		Consider the product $*$ on $A$ given by 
		\begin{align*}
			e_{ij}*e_{kl}=
			\begin{cases}
				e_{11}-e_{12}-e_{21}+e_{22}, & (i,l)=(1,2),\\
				0, & \text{otherwise}.
			\end{cases}
		\end{align*}
		Then $(A*A)*A=A*(A*A)=\{0\}$, so that $(A,*)$ is associative. Observe that $*$ is of the form \cref{e_ij*e_kl=lb.e_il+r_il} for $\lb=0$ and $r_{il}=e_{11}-e_{12}-e_{21}+e_{22}\in\Ann(A)$ for $(i,l)=(1,2)$ and $r_{il}=0$ otherwise. Thus, $*$ is $(12)$-matching with $\cdot$, but it is not totally compatible with $\cdot$ by \cref{id-matching-iff-r_il=0}.
	\end{exm}
	
	However, in some cases $\Ann(A)=\{0\}$ and we have the following result.
	
	\begin{prop}
		Let $|I|=1$ or $|J|=1$ (so that $S$ is a right zero or a left zero semigroup, respectively). For an associative product $*$ on $A$ the following are equivalent:
		\begin{enumerate}
			\item\label{*-(12)-matching-right-zero} $*$ is $(12)$-matching with $\cdot$;
			\item\label{*-interchangeable-right-zero} $*$ is interchangeable with $\cdot$;
			\item\label{*-totally-compatible-right-zero} $*$ is totally compatible with $\cdot$;
			\item\label{*-mutation-right-zero} $*$ is determined by some $\vf\in\G(A)$.
		\end{enumerate}
	\end{prop}
	\begin{proof}
		In view of \cref{interchangeable-is-totally-comp-A^2=A,totally-comp-rect-band} we only need to show \cref{*-(12)-matching-right-zero}$\iff$\cref{*-totally-compatible-right-zero}. If $I=\{i_0\}$, then $e_{i_0j}\cdot e_{i_0k}=e_{i_0k}$ for all $j,k\in J$, whence $e_{i_0j}\cdot a=a$ for all $a\in A$, so that each $e_{i_0j}$ is a left unit of $A$, and the result follows by \cref{(12)-matching-is-totally-comp-left-right-unit} (or from \cref{id-matching-iff-r_il=0}, since the existence of a left unit implies $\Ann(A)=\{0\}$). Whenever $J=\{j_0\}$, each $e_{ij_0}$ is a right unit of $A$, and we apply \cref{(12)-matching-is-totally-comp-left-right-unit} again.
	\end{proof}
	
	\subsection{Algebras with enough idempotents}\label{sec-enough-idempotents}
	
	An associative algebra $A$ is said to have \textit{enough idempotents}~\cite{Wisbauer} if there is a family $E$ of orthogonal idempotents $\{e_i\}_{i\in I}\sst A$ such that $A=\bigoplus_{i\in I}e_iA=\bigoplus_{i\in I}Ae_i$. It follows that $A=\bigoplus_{i,j\in I}e_iAe_j$. Observe that $A$ is an idempotent algebra, and it is unital if and only if $|I|<\infty$, in which case $\sum_{i\in I} e_i$ is the unit of $A$. 
	
	\begin{exm}
		The following non-unital associative $K$-algebras have enough idempotents.
		\begin{enumerate}
			\item The algebra of infinite matrices over $K$ with a finite number of nonzero entries. The corresponding idempotents are the matrix units $E_{ii}$.
			\item The path algebra of a quiver with infinite number of vertices. The corresponding idempotents are the trivial paths $e_x$, where $x$ is a vertex.
			\item The direct sum of an infinite family of unital algebras. The corresponding idempotents are the units of the direct summands. Any algebra with enough idempotents such that $E\sst C(A)$ is of this form.
		\end{enumerate}
	\end{exm}
	
	Let $(A,\cdot)$ be an algebra with enough idempotents. Given $a\in A$, we write $a=\sum_{i,j\in I}a_{ij}$, where $a_{ij}=e_iae_j\in e_iAe_j$. Define $M$ to be the $K$-space $\prod_{i,j\in I}e_iAe_j$ whose elements will be denoted by $m=\prod_{i,j\in I}m_{ij}$ with $m_{ij}\in e_iAe_j$. We will identify $A$ with the subspace of $M$ consisting of $m\in M$ with finite $\supp(m):=\{(i,j)\in I^2\mid m_{ij}\ne 0\}$.
	
	\begin{prop}
		The algebra structure on $A$ extends to an $A$-bimodule structure on $M$ by means of
		\begin{align}\label{am-and-ma}
			am=\prod_{i,j\in I}\left(\sum_{k\in I}a_{ik}m_{kj}\right)\text{ and }ma=\prod_{i,j\in I}\left(\sum_{k\in I}m_{ik}a_{kj}\right)
		\end{align}
		for all $a\in A$ and $m\in M$. Moreover, $AMA=A$.
	\end{prop}
	\begin{proof}
		The operations \cref{am-and-ma} are clearly well-defined and bilinear. For all $a,b\in A$ and $m\in M$ we have 
		\begin{align*}
			(a(bm))_{ij}&=\sum_{k\in I}a_{ik}(bm)_{kj}=\sum_{k\in I}\sum_{l\in I}a_{ik}b_{kl}m_{lj}\\
			&=\sum_{l\in I}\sum_{k\in I}a_{ik}b_{kl}m_{lj}=\sum_{l\in I}(ab)_{il}m_{lj}=((ab)m)_{ij},
		\end{align*}
		and similarly one proves that $((ma)b)_{ij}=(m(ab))_{ij}$. Finally, 
		\begin{align}\label{(amb)_ij=sum-a_ik.m_kl.b_lj}
			(a(mb))_{ij}=\sum_{k\in I}\sum_{l\in I}a_{ik}m_{kl}b_{lj}=\sum_{l\in I}\sum_{k\in I}a_{ik}m_{kl}b_{lj}=((am)b)_{ij}.
		\end{align}
		
		Since $I_1:=\{i\in I\mid\exists k\in I:\ a_{ik}\ne 0\}$ and $I_2:=\{j\in I\mid\exists l\in I:\ b_{lj}\ne 0\}$ are finite, then $\supp(amb)$ is finite as a subset of $I_1\times I_2$ by \cref{(amb)_ij=sum-a_ik.m_kl.b_lj}. Thus, $AMA\sst A$. The converse inclusion is obvious because $a_i=e_i\cdot e_i\cdot a_i$ for all $a_i\in e_iA$, so that $A=A\cdot A\cdot A\sst AMA$.
	\end{proof}

	\begin{lem}\label{e_i*e_j-in-e_iAe_j}
		Let $*$ be an associative product on $A$. If $*$ is $\id$-matching with $\cdot$, then $e_i*e_j\in e_iAe_j$ for all $i,j\in I$. 
	\end{lem}
	\begin{proof}
		Since $e_i\cdot e_i=e_i$, then $e_i*e_j = (e_i\cdot e_i)*e_j = e_i\cdot(e_i*e_j)$. Similarly, $e_j\cdot e_j=e_j$ yields $e_i*e_j=e_i*(e_j\cdot e_j)=(e_i*e_j)\cdot e_j$. Hence $e_i*e_j=e_i\cdot(e_i*e_j)\cdot e_j\in e_iAe_j$.
	\end{proof}
	
	We thus define 
	\begin{align}\label{m=sum-e_i*e_j}
		x=\prod_{i,j\in I}(e_i*e_j)\in M.
	\end{align}

	\begin{lem}\label{a_ij*a_kl=a_ij.m.a_kl}
		Let $*$ be an associative product on $A$. If $*$ is $\id$-matching with $\cdot$, then for all $a_{ij}\in e_iAe_j$ and $a_{kl}\in e_kAe_l$ we have $a_{ij}*a_{kl}=a_{ij}xa_{kl}$. 
	\end{lem}
	\begin{proof}
		By \cref{m=sum-e_i*e_j,(amb)_ij=sum-a_ik.m_kl.b_lj}
		\begin{align*}
			a_{ij}*a_{kl}&=(a_{ij}\cdot e_j)*a_{kl}=a_{ij}\cdot (e_j*a_{kl})=a_{ij}\cdot (e_j*(e_k\cdot a_{kl}))\\
			&=a_{ij}\cdot ((e_j*e_k)\cdot a_{kl})=a_{ij}\cdot x_{jk}\cdot a_{kl}=a_{ij}xa_{kl}.
			\qedhere
		\end{align*}
	\end{proof}
	
	\begin{prop}\label{id-matching-enough-idemp}
		The $\id$-matching structures on $(A,\cdot )$ are exactly the products $*$ on $A$ of the form 
		\begin{align}\label{a*b=a.m.b}
			a*b=axb
		\end{align}
		for all $a,b\in A$.
	\end{prop}
	\begin{proof}
		The fact that any $\id$-matching with $\cdot$ product $*$ on $A$ has form \cref{a*b=a.m.b} is a consequence of \cref{a_ij*a_kl=a_ij.m.a_kl}. Conversely, let $*$ be given by \cref{a*b=a.m.b}. Then \[(a*b)\cdot c=(axb)\cdot c=ax(b\cdot c)=a*(b\cdot c)\] and similarly \[a\cdot(b*c)=a\cdot(bxc)=(a\cdot b)xc=(a\cdot b)*c.\qedhere\]
	\end{proof}
	
	\begin{lem}\label{centroid-of-oplus-e_iAe_j}
		Let $\vf:A\to A$ be a linear map. Then $\vf\in\G(A)$ if and only if there exists $m\in M$ such that 
		\begin{align}\label{vf(a)=ma=am}
			\vf(a)=ma=am
		\end{align}
		for all $a\in A$.
	\end{lem}
	\begin{proof}
		Let $\vf\in\G(A)$. Then $\vf(e_i)=\vf(e_i\cdot e_i)=e_i\vf(e_i)=\vf(e_i)e_i$, so that $\vf(e_i)\in e_iAe_i$. Define $m=\prod_{i\in I}\vf(e_i)\in M$. For any $a_{ij}\in e_iAe_j$ we have \[\vf(a_{ij})=\vf(e_i\cdot a_{ij})=\vf(e_i)a_{ij}=m_{ii}a_{ij}=ma_{ij}.\] On the other hand, 
		\[\vf(a_{ij})=\vf(a_{ij}\cdot e_j)=a_{ij}\vf(e_j)=a_{ij}m_{jj}=a_{ij}m.\]
		Thus, \cref{vf(a)=ma=am} holds for all $a\in A$.
		
		Conversely, assume \cref{vf(a)=ma=am}. Then \[\vf(a)\cdot b=(ma)\cdot b=m(a\cdot b)=\vf(a\cdot b) \quad \textup{and} \quad a\cdot \vf(b)=a\cdot bm=(a\cdot b)m=\vf(a\cdot b),\]
		for all $a,b\in A$, so $\vf\in\G(A)$.
	\end{proof}
	
	\begin{lem}\label{e_i*e_i-in-e_iAe_i}
		Let $*$ be an associative product on $A$. If $*$ is $(12)$-matching with $\cdot$, then $e_i*e_j=0$ for all $i\ne j$ and $e_i*e_i\in e_iAe_i$ for all $i\in I$.
	\end{lem}
	\begin{proof}
		Let $i\ne j$. Then $e_i*e_j = (e_i\cdot e_i)*e_j = e_i*(e_i\cdot e_j)=0$. It follows that $(e_i*e_i)\cdot e_j=e_i\cdot (e_i*e_j)=0$ and $e_j\cdot (e_i*e_i)=(e_j*e_i)\cdot e_i=0$, so that $e_i*e_i\in e_iAe_i$.
	\end{proof}
	
	We thus define as in \cref{m=sum-e_i*e_j}
	\begin{align}\label{m=sum-e_i*e_i}
		x=\prod_{i\in I}(e_i*e_i)\in M.
	\end{align}
	
	\begin{lem}\label{prod(e_i*e_i)a=aprod(e_i*e_i)}
		Let $*$ be an associative product on $A$. If $*$ is $(12)$-matching with $\cdot$, then
		\begin{align}\label{m.a=a.m}
			xa=ax
		\end{align} 
		for all $a\in A$.
	\end{lem}
	\begin{proof}
		We have $e_k\cdot (e_i*a_{ij})=(e_k*e_i)\cdot a_{ij}=0$ for all $a_{ij}\in e_iAe_j$ and $k\ne i$ by \cref{e_i*e_i-in-e_iAe_i}, so that 
		\begin{align}\label{e_i*a_ij=ma_ij}
			e_i*a_{ij}=e_i\cdot (e_i*a_{ij})=(e_i*e_i)\cdot a_{ij}=x_{ii}a_{ij}=xa_{ij}.
		\end{align}
		Similarly, $(a_{ij}*e_j)\cdot e_k=a_{ij}\cdot(e_j*e_k)=0$ for all $a_{ij}\in e_iAe_j$ and $k\ne j$, so that 
		\begin{align}\label{a_ij*e_j=a_ijm}
			a_{ij}*e_j=(a_{ij}*e_j)\cdot e_j=a_{ij}\cdot(e_j*e_j)=a_{ij}x_{jj}=a_{ij}x.
		\end{align}
		But $e_i*a_{ij}=e_i*(a_{ij}\cdot e_j)=(e_i\cdot a_{ij})*e_j=a_{ij}*e_j$, whence \cref{m.a=a.m} by \cref{e_i*a_ij=ma_ij,a_ij*e_j=a_ijm}.
	\end{proof}
	
	\begin{lem}\label{a_ij*a_jk=a_ij.m.a_jk}
		Let $*$ be an associative product on $A$. If $*$ is $(12)$-matching with $\cdot$, then there exists $\vf\in\G(A)$ such that for all $a_{ij}\in e_iAe_j$ and $a_{kl}\in e_kAe_l$ we have 
		\begin{align}\label{a_ij*a_kl=vf(a_ij.a_kl)}
			a_{ij}*a_{kl}=\vf(a_{ij}\cdot a_{kl}).
		\end{align}
	\end{lem}
	\begin{proof}
		Let $\vf(a)=xa$, where $x$ is given by \cref{m=sum-e_i*e_i}. By \cref{centroid-of-oplus-e_iAe_j,prod(e_i*e_i)a=aprod(e_i*e_i)} we have $\vf\in\G(A)$. We are left to prove \cref{a_ij*a_kl=vf(a_ij.a_kl)}. If $j\ne k$, then 
		\[a_{ij}*a_{kl}=(a_{ij}\cdot e_j)*a_{kl}=a_{ij}*(e_j\cdot a_{kl})=0=\vf(a_{ij}\cdot a_{kl}).\] Now, by \cref{m.a=a.m,a_ij*e_j=a_ijm} we conclude that \[a_{ij}*a_{jk}=a_{ij}*(a_{jk}\cdot e_k)=(a_{ij}\cdot a_{jk})*e_k=(a_{ij}\cdot a_{jk})x=\vf(a_{ij}\cdot a_{jk}).\qedhere\]
	\end{proof}
	
	\begin{prop}\label{comp-prod-A-with-enough-idemp}
		Let $*$ be an associative product on $A$. Then the following are equivalent:
		\begin{enumerate}
			\item\label{*-(12)-matching} $*$ is $(12)$-matching with $\cdot$;
			\item\label{*-interchangeable} $*$ is interchangeable with $\cdot$;
			\item\label{*-totally-compatible} $*$ is totally compatible with $\cdot$;
			\item\label{*-mutation} $*$ is determined by some $\vf\in\G(A)$.
		\end{enumerate}
	\end{prop}
	\begin{proof}
		The implications \cref{*-mutation}$\impl$\cref{*-totally-compatible} and \cref{*-totally-compatible}$\impl$\cref{*-(12)-matching} are obvious. The implication \cref{*-(12)-matching}$\impl$\cref{*-mutation} is a consequence of \cref{a_ij*a_jk=a_ij.m.a_jk} and the equivalence \cref{*-interchangeable}$\iff$\cref{*-totally-compatible} is \cref{interchangeable-is-totally-comp-A^2=A}.
	\end{proof}

	\section{\texorpdfstring{$\sg$}{σ}-matching and interchangeable structures on free non-unital associative algebras}\label{sec-free-alg}
	
	We begin with a general result.
	\begin{prop}\label{(12)-match=>id-match-without-0-div}
		Let $(A,\cdot)$ be an associative algebra without zero divisors. Then for an associative product $*$ on $A$ the following are equivalent:
		\begin{enumerate}
			\item\label{*-(12)-matching-without-0-div} $*$ is $(12)$-matching with $\cdot$;
			\item\label{*-interchangeable-without-0-div} $*$ is interchangeable with $\cdot$;
			\item\label{*-totally-compatible-without-0-div} $*$ is totally compatible with $\cdot$.
		\end{enumerate}
	\end{prop}
	\begin{proof}
		\cref{*-(12)-matching-without-0-div}$\impl$\cref{*-totally-compatible-without-0-div}. Let $*$ be $(12)$-matching with $\cdot$. Then for all $a,b,c\in A$ we have 
		\[a\cdot (a*(b\cdot c))=(a*a)\cdot(b\cdot c)=((a*a)\cdot b)\cdot c=(a\cdot(a*b))\cdot c=a\cdot((a*b)\cdot c).\]
		Since $(A,\cdot)$ has no zero divisors, $a\cdot (a*(b\cdot c))=a\cdot((a*b)\cdot c)$ implies $a*(b\cdot c)=(a*b)\cdot c$. Thus, $*$ is totally compatible with $\cdot$ by \cref{tot-com-equiv-cond}\cref{cdot_1-and-cdot_2-sg-match-for-all-sg}.
		
		\cref{*-interchangeable-without-0-div}$\impl$\cref{*-totally-compatible-without-0-div}. Let $*$ be interchangeable with $\cdot$. Then for all $a,b,c\in A$ we have 
		\[a\cdot ((a\cdot b)*c)=a*((a\cdot b)\cdot c)=a*(a\cdot (b\cdot c))=a\cdot(a*(b\cdot c)).\]
		As above, $a\cdot ((a\cdot b)*c)=a\cdot(a*(b\cdot c))$ yields $(a\cdot b)*c=a*(b\cdot c)$, because $(A,\cdot)$ has no zero divisors. Thus, $*$ is totally compatible with $\cdot$ by \cref{tot-com-equiv-cond}\cref{cdot_1-and-cdot_2-sg-match-for-some-sg}.
		
		The implications \cref{*-totally-compatible-without-0-div}$\impl$\cref{*-(12)-matching-without-0-div} and \cref{*-totally-compatible-without-0-div}$\impl$\cref{*-interchangeable-without-0-div} are trivial.
	\end{proof}
	Clearly, $\id$-matching structures on an algebra without zero divisors are not in general totally compatible: as a counterexample it is enough to take a non-commutative associative unital division algebra (see \cref{mutation-by-non-central-element}).
	
	Free algebras are a classical example of algebras with no zero divisors. In view of \cref{comp-prod-unital-algebra}, we will only be interested in the non-unital case.
	
	\subsection{The free non-unital associative algebra}
	Fix a set of variables $X$ with $|X|>1$. Let $X^*$ (resp. $X^+$) be the free monoid (resp. free semigroup) over $X$. The elements of $X^*$ (resp. $X^+$) are all the words (resp. non-empty words) over $X$, where the product of any two words $w_1$ and $w_2$ is their concatenation $w_1w_2$. Let $(K\gen X,\cdot)$ (resp. $(K \gen X^+,\cdot)$) be the free associative (resp. free \textit{non-unital} associative) $K$-algebra over $X$. Observe that $K \gen X$ (resp. $K \gen X^+$) is the semigroup $K$-algebra of $X^*$ (resp. $X^+$). Recall from~\cite[6.1]{Bresar-noncomm-alg} that $K\gen X$ is a domain. Clearly, $K \gen X^+$ is an ideal in $K\gen X$ and 
	\begin{align}\label{K<X>^+=oplus-x-K<X>}
		K\gen X^+=\bigoplus_{x\in X}x\cdot K\gen X=\bigoplus_{x\in X}K\gen X\cdot x.
	\end{align}

	\begin{prop}\label{id-matching-free-noncomm}
		The $\id$-matching structures on $(K\gen X^+,\cdot)$ are in a one-to-one correspondence with maps $\star:X\times X\to K\gen X^+$, 
		\begin{align}\label{x-star-y=sum-u.R_xyu=sum-L_xyv.v}
			(x,y)\mapsto x\star y=\sum_{u\in X}u\cdot R_{x,y,u}=\sum_{v\in X}L_{x,y,v}\cdot v,
		\end{align} 
		where $R_{x,y,u},L_{x,y,v}\in K\gen X$ and
		\begin{align}\label{(x-star-y)*z=x*(y-star-z)}
			\sum_{v\in X}L_{x,y,v}\cdot(v\star z)=\sum_{u\in X}(x\star u)\cdot R_{y,z,u}
		\end{align}
		for all $x,y,z\in X$.
	\end{prop}
	\begin{proof}
		Let $*$ be an $\id$-matching structure on $(K\gen X^+,\cdot)$. 
		
		Given $x,y\in X$, write $x*y=\sum_{u\in X}u\cdot R_{x,y,u}=\sum_{v\in X}L_{x,y,v}\cdot v$ according to \cref{K<X>^+=oplus-x-K<X>}. Then 
		\[(x*y)*z=\sum_{v\in X}(L_{x,y,v}\cdot v)*z=\sum_{v\in X}L_{x,y,v}\cdot (v*z)\] and 
		\[x*(y*z)=\sum_{u\in X}x*(u\cdot R_{y,z,u})=\sum_{u\in X}(x*u)\cdot R_{y,z,u}\] for all $x,y,z\in X$. Thus, defining $\star$ to be the restriction of $*$ to $X\times X$, we get \cref{x-star-y=sum-u.R_xyu=sum-L_xyv.v,(x-star-y)*z=x*(y-star-z)} by the associativity of $*$.
		
		Conversely, let $\star:X\times X\to K\gen X^+$ satisfying \cref{x-star-y=sum-u.R_xyu=sum-L_xyv.v,(x-star-y)*z=x*(y-star-z)}. Given $a,b\in X^+$, there are unique $x,y\in X$ and $a_1,b_1\in X^*$, such that $a=a_1x$ and $b=yb_1$. Then define 
		\begin{align}\label{(a_1x)*(yb_1)=a_1(x*y)b_1}
			a*b=a_1\cdot(x\star y)\cdot b_1.
		\end{align}
		In particular, $x*y=x\star y$ for all $x,y\in X$. The product \cref{(a_1x)*(yb_1)=a_1(x*y)b_1} uniquely extends to a bilinear product $*$ on the whole $K\gen X^+$. Taking additionally $c\in X^+$, since $bc=yb_1c$, by \cref{(a_1x)*(yb_1)=a_1(x*y)b_1} we have
		\begin{align}\label{(a*b)c-is-a*(bc)}
			(a*b)\cdot c=(a_1\cdot (x\star y)\cdot b_1)\cdot c=a_1\cdot (x\star y)\cdot (b_1c)=a*(bc).
		\end{align}
		Moreover, if $b=b_2u$ and $c=zc_1$ with $u,z\in X$ and $b_2,c_1\in X^*$, then 
		\begin{align}\label{(ab)*c-is-a(b*c)}
			(ab)*c=(ab_2u)*(zc_1)=(ab_2)\cdot (u\star z)\cdot c_1=a\cdot (b_2\cdot (u\star z)\cdot c_1)=a\cdot (b*c).
		\end{align}
		
		Thus, we are left to prove that $*$ is associative. Observe from \cref{(a*b)c-is-a*(bc),(ab)*c-is-a(b*c),x-star-y=sum-u.R_xyu=sum-L_xyv.v} that \cref{(x-star-y)*z=x*(y-star-z)} is equivalent to
		\begin{align}\label{(x-star-y)*z-is-x*(y-star-z)}
			(x\star y)*z=x*(y\star z)
		\end{align}
		for all $x,y,z\in X$. Let $a=a_1x$, $b=yb_1$ and $c=zc_1$ with $x,y,z\in X$ and $a_1,b_1,c_1\in X^*$. Then by \cref{(a*b)c-is-a*(bc),(ab)*c-is-a(b*c)} we have
		\begin{align}\label{(a*b)*c=a_1(((x*y)b_1)*z)c_1}
			(a*b)*c=(a_1\cdot (x\star y)\cdot b_1)*c=a_1\cdot(((x\star y)\cdot b_1)*z)\cdot c_1.
		\end{align}
		
		\textit{Case 1.} $b_1$ is empty. Then using \cref{(a_1x)*(yb_1)=a_1(x*y)b_1,(a*b)*c=a_1(((x*y)b_1)*z)c_1,(x-star-y)*z-is-x*(y-star-z),(ab)*c-is-a(b*c)} we have \[(a*b)*c=a_1\cdot((x\star y)*z)\cdot c_1=a_1\cdot(x*(y\star z))\cdot c_1=(a_1x)*((y\star z)\cdot c_1)=a*(b*c).\]
		
		\textit{Case 2.} $b_1=b_2u$ for some $u\in X$ and $b_2\in X^*$. Then using \cref{(a_1x)*(yb_1)=a_1(x*y)b_1,(a*b)*c=a_1(((x*y)b_1)*z)c_1,(a*b)c-is-a*(bc),(ab)*c-is-a(b*c)} we have 
		\begin{align*}
			(a*b)*c&=a_1\cdot(((x\star y)\cdot b_2u)*z)\cdot c_1=a_1\cdot((x\star y)\cdot b_2\cdot (u*z))\cdot c_1\\ & =(a_1\cdot(x\star y))\cdot(b_2\cdot(u\star z)\cdot c_1)=(a*y)\cdot(b_1*c)\\ &=a*(y\cdot(b_1*c))=a*((yb_1)*c)=a*(b*c).\qedhere\end{align*}
	\end{proof}
	
	\begin{rem}
		We are unable to describe the maps $\star:X\times X\to K\gen X^+$ satisfying \cref{x-star-y=sum-u.R_xyu=sum-L_xyv.v,(x-star-y)*z=x*(y-star-z)} in simpler terms, so let us just give some natural examples. Any associative product $\star:X\times X\to X$, i.e. a semigroup structure on $X$, is of this form. On the other hand, any mutation $x\star y=x\cdot_p y$, where $p\in K\gen X$, is also of this form. We suspect that there also exist examples of more complicated maps $\star$.
	\end{rem}
	
	\begin{lem}\label{x*y=lb.xy-K<X>^+}
		Let $*$ be a totally compatible structure on $(K\gen X^+,\cdot)$. Then there exists $\lb\in K$ such that $x*y=\lb(x\cdot y)$ for all $x,y\in X$.
	\end{lem}
	\begin{proof}
		For all $x\ne y$ in $X$ we have 
		\begin{align}\label{x.(y*y)=(x*x).y}
			x\cdot(x*y)=(x*x)\cdot y,
		\end{align}
		so $x*y\in K\gen X\cdot y$ and $x*x\in x\cdot K\gen X$. Similarly, 
		\begin{align}\label{x.(y*y)=(x*y).y}
			x\cdot (y*y)=(x*y)\cdot y
		\end{align}
		implies $x*y\in x\cdot K\gen X$ and $x*x\in K\gen X\cdot x$. Thus, there exist $p_{x,y}\in K\gen X$, $x\ne y$, such that 
		\begin{align}\label{x*y=x.p_xy.y}
			x*y=x\cdot p_{x,y}\cdot y,
		\end{align}
		and $p_x,q_x\in K\gen X$ such that 
		\begin{align}\label{x*x=p_x.x=x.q_x}
			x*x=p_x\cdot x=x\cdot q_x.
		\end{align}
		It follows from $(x*x)\cdot x=x\cdot(x*x)$ that $p_x=q_x$, which commutes with $x$ by \cref{x*x=p_x.x=x.q_x}. By Bergman's centralizer theorem~\cite{Bergman69} (or directly by induction on the degree of $p_x$) one concludes that $p_x\in K\gen x$. Now, \cref{x.(y*y)=(x*x).y} gives 
		\begin{align}\label{x.p_xy=q_x}
			x\cdot p_{x,y}=q_x
		\end{align}
		and \cref{x.(y*y)=(x*y).y} gives 
		\begin{align}\label{p_xy.y=p_y}
			p_{x,y}\cdot y=p_y,
		\end{align}
		so that $p_{x,y}\in K\gen x\cap K\gen y=K$. It also follows from \cref{x.p_xy=q_x} that $p_{x,y}=p_{x,v}$ for all $y,v\ne x$. Similarly, \cref{p_xy.y=p_y} gives $p_{x,y}=p_{u,y}$ for all $x,u\ne y$. Let $u\ne v$ and $x\ne y$. If $u\ne y$, then $p_{u,v}=p_{u,y}=p_{x,y}$. If $v\ne x$, then $p_{u,v}=p_{x,v}=p_{x,y}$. If $(u,v)=(y,x)$, then $(x*y)\cdot x=x\cdot(u*v)$ yields $p_{x,y}\cdot y=y\cdot p_{u,v}$, whence $p_{x,y}=p_{u,v}$ because $p_{x,y},p_{u,v}\in K$. Thus, $p_{x,y}$ is a scalar $\lb\in K$ that does not depend on $x$ and $y$, so the result follows by \cref{x*y=x.p_xy.y,x.p_xy=q_x,x*x=p_x.x=x.q_x}.
	\end{proof}
	
	The result of the next corollary is probably well-known, but we couldn't find a reference.
	\begin{cor}
		We have $\G(K\gen X^+)=K$.
	\end{cor}
	\begin{proof}
		Let $\vf\in\G(K\gen X^+)$ and $*_\vf$ be the corresponding totally compatible product as in \cref{a*b=vf(ab)-tot-comp}. Then there exists $\lb\in K$ such that $\vf(x\cdot y)=x*_\vf y=\lb(x\cdot y)$ for all $x,y\in X$  by \cref{x*y=lb.xy-K<X>^+}. On the other hand, $\vf(x\cdot y)=\vf(x)\cdot y$, whence $\vf(x)=\lb x$ for all $x\in X$. Let now $a=xa_1\in X^+$, where $x\in X$ and $a_1\in X^+$. Then $\vf(a)=\vf(xa_1)=\vf(x)\cdot a_1=\lb x\cdot a_1=\lb a$. By linearity, $\vf=\lb\id$.
	\end{proof}
	
	\begin{prop}\label{totally-comp-free-noncomm}
		Let $*$ be an associative product on $K\gen X^+$. Then the following are equivalent:
		\begin{enumerate}
			\item\label{*-(12)-matching-K<x>^+} $*$ is $(12)$-matching with $\cdot$;
			\item\label{*-interchangeable-K<x>^+} $*$ is interchangeable with $\cdot$;
			\item\label{*-totally-compatible-K<x>^+} $*$ is totally compatible with $\cdot$;
			\item\label{a*b=vf(ab)-K<x>^+} $*$ is determined by some $\vf\in \G(K\gen X^+)$.
		\end{enumerate}
	\end{prop}
	\begin{proof}
		In view of \cref{(12)-match=>id-match-without-0-div,a*b=vf(ab)-tot-comp} we only need to prove \cref{*-totally-compatible-K<x>^+}$\impl$\cref{a*b=vf(ab)-K<x>^+}. Let $*$ be totally compatible with $\cdot$. By \cref{x*y=lb.xy-K<X>^+} there exists $\vf=\lb\id\in\G(K\gen X^+)$, such that $x*y=\vf(x\cdot y)$ for all $x,y\in X$. Then, given $a=a_1x$ and $b=yb_1$ with $x,y\in X$ and $a_1,b_1\in X^*$, we have 
		\begin{align*}
			a*b
			&=a_1x*yb_1\\
			&=a_1\cdot(x*y)\cdot b_1\\
			&=a_1\cdot\lb(x\cdot y)\cdot b_1\\
			&=\lb(a_1x\cdot yb_1)\\
			&=\lb(a\cdot b)\\
			&=\vf(a\cdot b).
			%\vf(a\cdot b)&=\vf(a_1x\cdot yb_1)=a_1\cdot\vf(x\cdot y)\cdot b_1=a_1\cdot\lb(x\cdot y)\cdot b_1\\&=\lb(a_1x\cdot yb_1)=\lb(a\cdot b)=\vf(a\cdot b).
			\qedhere
		\end{align*}
	\end{proof}
	
	\subsection{The free non-unital commutative associative algebra}
	
	Fix a non-empty set of variables $X$. Let $(K[X],\cdot)$ (resp. $(K[X]^+,\cdot)$) be the free commutative (resp. free \textit{non-unital} commutative) associative $K$-algebra over $X$. Observe that $K[X]$ is the algebra of polynomials in $x\in X$ over $K$ and $K[X]^+$ is its ideal of polynomials with zero constant term. It is clear that $K[X]$ is an integral domain and $K[X]^+=\sum_{x\in X}xK[X]$ (the sum is not direct for $|X|>1$).

	\begin{lem}\label{descr-G(K[X])^+}
		Let $\vf:K[X]^+\to K[X]^+$ be a linear map. Then $\vf\in\G(K[X]^+)$ if and only if there exists $p\in K[X]$ such that 
		\begin{align}\label{vf(a)-is-pa}
			\vf(a)=pa
		\end{align}
		for all $a\in K[X]^+$.
	\end{lem}
	\begin{proof}
		It is obvious that a map of the form \cref{vf(a)-is-pa} belongs to $\G(K[X]^+)$. Conversely, let $\vf\in\G(K[X]^+)$. 
		
		\textit{Case 1.} $X=\{x\}$. Then defining $p:=\frac{\vf(x)}x\in K[X]$ we have 
		\[\vf(x^n)=x^{n-1}\vf(x)=x^{n-1}\cdot xp=x^np\]
		for all $n\in\Z_+$, whence \cref{vf(a)-is-pa} by linearity.
		
		\textit{Case 2.} $|X|>1$. Then choosing $x\ne y$ in $X$ we have 
		\begin{align}\label{xvf(y)=vf(xy)=yvf(x)}
			x\vf(y)=\vf(xy)=y\vf(x).
		\end{align}
		Since $x\ne y$, it follows that $\vf(x)$ is divisible by $x$ (in $K[X]$), i.e. there exists $p_x\in K[X]$ such that $\vf(x)=xp_x$. Then \cref{xvf(y)=vf(xy)=yvf(x)} gives $xyp_y=xyp_x$, so $p_x=p_y$, which will be denoted by $p$. We thus obtain \cref{vf(a)-is-pa} for $a$ being a variable. Now, for any $x\in X$ and $a_1\in K[X]^+$ we have $\vf(xa_1)=\vf(x)a_1=pxa_1$, so \cref{vf(a)-is-pa} holds for monomials of higher degree and thus for all $a\in K[X]^+$ by linearity.
	\end{proof}

	\begin{lem}\label{(xa_1)*(yb_1)-in-terms-of-(x*y)}
		Let $*$ be an $\id$-matching structure on $(K[X]^+,\cdot)$. Then for all $x,y\in X$ and $a_1,b_1\in K[X]$ we have
		\begin{align*}%\label{(xa_1)*(yb_1)=a_1b_1(x*y)}
			(xa_1)*(yb_1)=(x*y)a_1b_1.
		\end{align*}
	\end{lem}
	\begin{proof}
		We have $(xa_1)*(yb_1)=(a_1x)*(yb_1)=a_1\cdot(x*(yb_1))=a_1\cdot(x*y)\cdot b_1=(x*y)a_1b_1$.
	\end{proof}
	
	\begin{prop}\label{comp-prod-on-K[x]}
		Let $X=\{x\}$ and $*$ be an associative product on $K[X]^+$. Then the following are equivalent:
		\begin{enumerate}
			\item\label{*-id-matching-K[x]^+} $*$ is $\id$-matching with $\cdot$;
			\item\label{*-(12)-matching-K[x]^+} $*$ is $(12)$-matching with $\cdot$;
			\item\label{*-interchangeable-K[x]^+} $*$ is interchangeable with $\cdot$;
			\item\label{*-totally-compatible-K[x]^+} $*$ is totally compatible with $\cdot$;
			\item\label{a*b=a-over-x.p.b-over-x} there exists $p\in K[X]^+$ such that 
			\begin{align}\label{a*b=a_x.p.b_x}
				a*b=\frac{a\cdot b}{x^2}\cdot p
			\end{align}
			for all $a,b\in K[X]^+$.
		\end{enumerate}
	\end{prop}
	\begin{proof}
		\cref{*-id-matching-K[x]^+}$\impl$\cref{a*b=a-over-x.p.b-over-x}. Let $*$ be $\id$-matching with $\cdot$. Then by \cref{(xa_1)*(yb_1)-in-terms-of-(x*y)} we have 
		\[x^m*x^n=(x*x)x^{m+n-2}=\frac{x^m\cdot x^n}{x^2}(x*x).\] 
		Denoting $p:=x*x$, we get \cref{a*b=a_x.p.b_x} by bilinearity of $*$ and $\cdot$. 
		
		\cref{a*b=a-over-x.p.b-over-x}$\impl$\cref{*-totally-compatible-K[x]^+}. Suppose that $*$ is given by \cref{a*b=a_x.p.b_x}. Then $(a*b)*c=a*(b*c)=\frac{a\cdot b\cdot c}{x^4}\cdot p^2$ and 
		\[(a*b)\cdot c=a*(b\cdot c)=(a\cdot b)*c=a\cdot(b*c)=\frac{a\cdot b\cdot c}{x^2}\cdot p\]
		for all $a,b,c\in K[X]^+$.
		
		\cref{*-totally-compatible-K[x]^+}$\impl$\cref{*-id-matching-K[x]^+} is trivial. The remaining equivalences are \cref{(12)-match=>id-match-without-0-div}.
	\end{proof}
	
	\begin{rem}
		Observe by \cref{descr-G(K[X])^+} that the product \cref{a*b=a_x.p.b_x} from \cref{comp-prod-on-K[x]}\cref{a*b=a-over-x.p.b-over-x} is determined by some $\vf\in\G(K[X]^+)$ if and only if $\deg(p)\ge 2$.
	\end{rem}
	
	\begin{lem}\label{x*y-in-terms-of-x-cdot-y}
		Let $|X|>1$ and $*$ be an $\id$-matching structure on $(K[X]^+,\cdot)$. Then there exists $\vf\in \G(K[X]^+)$ such that for all $x,y\in X$:
		\begin{align}\label{x*y=xy-cdot-p}
			x*y=\vf(x\cdot y).
		\end{align}
	\end{lem}
	\begin{proof}
		For all $x\ne y$ in $X$ we have $x\cdot (y*y)=(xy)*y=y\cdot (x*y)$. It follows that $x*y$ is divisible by $x$ and $y*y$ is divisible by $y$ (in $K[X]$). Similarly, changing the brackets in $x*(yx)=x*(xy)$, we get 
		\begin{align}\label{(x*y)x=(x*x)y}
			(x*y)\cdot x=(x*x)\cdot y,
		\end{align}
		which shows that $x*y$ is divisible by $y$ (in $K[X]$). Thus, for all $x\in X$ there is $p_x\in K[X]$ and for all $x,y\in X$, $x\ne y$, there is $p_{x,y}\in K[X]$ such that 
		\begin{align*}%\label{x*x=xp_x-and-x*y=xyp_xy}
			x*x=xp_x\text{ and }x*y=xyp_{x,y}.
		\end{align*}
		Choosing another pair $u\ne v$ in $X$, we have 
		\[(xu)*(yv)=x\cdot (u*(yv))=x\cdot (u*v)\cdot y=xyuvp_{u,v}\] and similarly \[(xu)*(yv)=u\cdot (x*y)\cdot v=xyuvp_{x,y}.\] Thus, 
		\begin{align*}%\label{p_xy-is-p_uv}
			p_{x,y}=p_{u,v}=:p,
		\end{align*}
		and \cref{x*y=xy-cdot-p} holds for $x\ne y$ and $\vf(x)=px$. Moreover, \cref{(x*y)x=(x*x)y} shows that $xyp_x=x^2yp$, whence 
		\begin{align*}%\label{p_x-is-xp}
			p_x=xp,
		\end{align*}
		and \cref{x*y=xy-cdot-p} holds for $x=y$ and the same $\vf$ too.
	\end{proof}
	
	\begin{prop}\label{comp-prod-on-K[X]-with-|X|>1}
		Let $|X|>1$ and $*$ be an associative product on $K[X]^+$. Then the following are equivalent:
		\begin{enumerate}
			\item\label{*-id-matching-K[X]^+} $*$ is $\id$-matching with $\cdot$;
			\item\label{*-(12)-matching-K[X]^+} $*$ is $(12)$-matching with $\cdot$;
			\item\label{*-interchangeable-K[X]^+} $*$ is interchangeable with $\cdot$;
			\item\label{*-totally-compatible-K[X]^+} $*$ is totally compatible with $\cdot$;
			\item\label{*-mutation-by-p} $*$ is determined by some $\vf\in\G(K[X]^+)$.
		\end{enumerate}
	\end{prop}
	\begin{proof}
		\cref{*-id-matching-K[X]^+}$\impl$\cref{*-mutation-by-p}. Let $*$ be $\id$-matching with $\cdot$. In view of \cref{(xa_1)*(yb_1)-in-terms-of-(x*y),x*y-in-terms-of-x-cdot-y} there exists $\vf\in \G(K[X]^+)$ such that 
		\[(xa_1)*(yb_1)=(x*y)a_1b_1=\vf(x\cdot y)\cdot a_1b_1=\vf(xa_1\cdot yb_1)\] for all $x,y\in X$ and $a_1,b_1\in K[X]$. Then \cref{*-mutation-by-p} follows by bilinearity. 
		
		The implications \cref{*-mutation-by-p}$\impl$\cref{*-totally-compatible-K[X]^+} and \cref{*-totally-compatible-K[X]^+}$\impl$\cref{*-id-matching-K[X]^+} hold in any associative algebra. 
		
		The remaining equivalences are \cref{(12)-match=>id-match-without-0-div}.
	\end{proof}
	
	% Acknowledgments should be added at the end of this section (right before
	% the refences section) as a \subsection* (a subsection without a number):
	% \subsection*{Acknowledgments} ...
	
	\subsection*{Acknowledgements}
	This work was partially supported by CMUP, member of LASI, which is financed by national funds through FCT --- Fundação para a Ciência e a Tecnologia, I.P., under the project with reference UIDB/00144/2020. The author is grateful to Ivan Kaygorodov for useful comments on associative algebras and to the referees whose suggestions helped to correct several typos and improved the readability of the paper.
	
	%%% REFERENCES %%%
	%{\small\bibliography{Mykola}}
	% Please, do not change the above line and do not insert your references
	% into this file.  Instead, insert your references into the cimart.bib file.
	% See cimart.bib for further instructions.

	\EditInfo{July 29, 2024}{September 22, 2024}{Ivan Kaygorodov and David Towers}
\end{document}